\documentclass[a4paper]{amsart}
\usepackage{amssymb}
\usepackage{amsthm}
\usepackage{amsmath,amscd}
\usepackage[mathscr]{euscript}
\usepackage[all]{xy}
\usepackage{color}
\usepackage[utf8]{inputenc}
\usepackage{lmodern}
\usepackage[T1]{fontenc}
\usepackage[colorlinks=true,linkcolor=red,citecolor=blue]{hyperref}
\usepackage[textwidth=15cm,textheight=22cm,hcentering]{geometry}
\newtheorem{theo}{Theorem}[section]
\newtheorem{lemm}[theo]{Lemma}
\newtheorem{prop}[theo]{Proposition}
\newtheorem{coro}[theo]{Corollary}

\theoremstyle{definition}
\newtheorem{defi}[theo]{Definition}

\newtheorem{rem}[theo]{Remark}

\newtheorem*{theo*}{Theorem}

\numberwithin{equation}{section}

\newcommand{\cat}{\mathbf}
\newcommand{\id}{\mathrm{id}}

\newcommand{\Map}{\on{Map}}
\newcommand{\Hom}{\mathsf{Hom}}

\newcommand{\on}{\operatorname}
\renewcommand{\P}{\mathbf{P}}
\newcommand{\A}{\mathbf{A}}
\newcommand{\C}{\mathbf{C}}
\newcommand{\Q}{\mathbf{Q}}
\newcommand{\Z}{\mathbf{Z}}
\newcommand{\N}{\mathbf{N}}
\newcommand{\F}{\mathbf{F}}
\newcommand{\M}{\mathrm{M}}
\renewcommand{\H}{\mathrm{H}}

\newcommand{\MZ}{\mathsf{M}\Z}

\newcommand{\DM}{\mathbf{DM}}
\newcommand{\DMT}{\mathbf{DMT}}
\newcommand{\D}{\mathbf{D}}
\newcommand{\MS}{\mathbf{MS}}
\renewcommand{\S}{\mathbf{S}}
\newcommand{\gr}{\mathrm{gr}}

\renewcommand{\lim}{\mathop{\mathrm{lim}}}

\title{Motivic homological stability for configuration spaces of the line}

\author{Geoffroy Horel}

\begin{document}

\address{Max Planck Institute for Mathematics\\ Vivatsgasse 7\\ 53111 Bonn\\ Deutschland}
\email{geoffroy.horel@gmail.com}

\begin{abstract}
We lift the classical theorem of Arnol'd on homological stability for configuration spaces of the plane to the motivic world. More precisely, we prove that the schemes of unordered configurations of points in the affine line satisfy stability with respect to the motivic t-structure on mixed Tate motives.
\end{abstract}

\maketitle

\tableofcontents

\section{Introduction}

There are many examples in mathematics of a sequence of spaces
\[X_1\to X_2\to \ldots\to X_n\to\ldots\]
such that for any integer $i$, the induced sequence on homology groups
\[\H_i(X_1,\Z)\to \H_i(X_2,\Z)\to\ldots\to \H_i(X_n,\Z)\to\ldots\]
eventually becomes constant. Such a phenomenon is called homological stability. Let us mention a few examples: homological stability holds for the classifying spaces of the symmetric groups $\Sigma_n$ (Nakaoka), the classifying space of the mapping class groups $\Gamma_{g,1}$ of a surface of genus $g$ with one boundary component (Harer), the classifying space of automorphisms of free groups $\mathrm{Aut}(F_n)$ (Hatcher--Vogtmann), the space of unordered configurations of points in an open manifold (McDuff \cite{mcduffconfiguration}).

This paper is concerned with the case of configurations of points in the plane. In that case McDuff reduces to a theorem of Arnol'd of homological stability for classifying spaces of braid groups. Indeed, it is a classical fact that the space of unordered configurations of $n$ points in the plane is a model for the classifying space of the braid group on $n$ strands. There is a stabilizing map $B_n\to B_{n+1}$ which adjoins a trivial strand on the side of a braid with $n$-strands. The stability theorem for the braid groups is then given by the following theorem.

\begin{theo*}[Arnol'd, \cite{arnoldtopological}]\label{theo: arnold}
The stabilizing map induces an isomorphism
\[\H^i(B_{n+1},\Z)\cong \H^i(B_{n},\Z)\]
when $i<\lfloor{n/2}\rfloor+2$. 
\end{theo*}

The purpose of the present paper is to prove a motivic version of this result. In order to state our main theorem the first step is to construct algebro-geometric versions of the space of configurations in the plane. These are given by a sequence of schemes $C_n$ whose points over a commutative ring $R$ are the square-free degree $n$ monic polynomials over $R$. This choice is justified by the fact that the complex points of $C_n$ with the usual complex topology recovers the topological space $C_n(\mathbf{R}^2)$ of unordered configurations in the plane. Our motivic homological stability result has the following form.

\begin{theo*}[\ref{coro: proof of main}]\label{theo: main}
The motivic cohomology groups $\H^i(C_n,\Z(q))$ are independent of $n$ as long as $i<\lfloor{n/2}\rfloor+2$ and $n$ is at least $3$.
\end{theo*}

The functor $\H^i(-,\Z(q))$ are the motivic cohomology groups of Levine constructed in \cite{levinektheory} using a generalization of the cycle complex of Bloch. For us, they will be defined as 
\[\H^i(X,\Z(q))\cong \DM(\Z)(\M(-),\Z(q)[i])\]
where $\DM(\Z)$ denotes Spitzweck category of motives over $\on{Spec}(\Z)$ and $\M(-)$ is the functor which assigns to a smooth scheme over $\Z$ its motive. We refer the reader to Section \ref{section: mixed motives} for more details about this category.

This theorem follows from another theorem (Theorem \ref{theo: main bis}) which we will roughly state by saying that the collection of motives $\M(C_n)$ has homological stability with respect to the motivic t-structure. The existence of a motivic t-structure on the whole triangulated category of motives is one of the most important open conjecture in the area. However, such a t-structure has been constructed when we restrict to the full triangulated subcategory of mixed Tate motives provided that the base satisfies the Beilinson--Soul\'e vanishing conjecture \cite{levinetate}. This last conjecture is satisfied for $\on{Spec}(\Z)$ and the motives $\M(C_n)$ are mixed Tate motives (\ref{C and F are mixed Tate}). 

Contrary to the classical case, we were unable to construct stabilizing maps $C_n\to C_{n+1}$ in the unstable (or even stable) motivic homotopy category inducing this isomorphism. The existence of such maps is related to the existence of a motivic model of the operad of little $2$-disks. Evidence for the existence of such an object is given by the action of the absolute Galois group of $\Q$ (resp. the Tannakian Galois group of $\cat{MTM}(\Z)$) on the profinite completion (resp. the rational completion) of the operad of little $2$-disks \cite{horelprofinite,fressehomotopy}.

Our strategy can be described as follows:
\begin{itemize}
\item[-] First, we construct a scanning map $C_n\to F_n$ where $F_n$ is the scheme of pointed degree $n$ maps from $\P^1$ to $\P^1$ (Section \ref{section: scanning}).
\item[-] We prove that the motives $\M(C_n)$ and $ \M(F_n)$ are mixed Tate motives (\ref{C and F are mixed Tate}).
\item[-] We show that the motives $\M(F_n)$ have motivic homological stability, in the sense that certain natural maps $\M(F_n)\to \M(F_{n+1})$ have a fiber whose connectivity tends to $\infty$. The connectivity is measured with respect to the motivic $t$-structure on mixed Tate motives (\ref{prop: stability for F}).
\item[-] We show that the  connectivity of the fiber of the map induced by the scanning map $\M(C_n)\to  \M(F_n)$ tends to $\infty$ with respect to the motivic $t$-structure on mixed Tate motives (\ref{theo: main bis}).
\item[-] The previous two facts together imply the desired theorem.
\end{itemize}
This strategy should be compared to the one used by Bendersky and Miller in \cite{benderskylocalization} in the case of configuration spaces in a closed manifold. In that case the stabilization maps are also unavailable and the authors study homological stability using the scanning map.

By definition of the category of motives, this theorem implies homological stability for the Betti cohomology of $C_n\times_{\Z}\Q$ (which is exactly Arnol'd theorem) but also for the \'etale, Hodge or de Rham cohomology of $C_n\times_{\Z}\Q$. For $C_n\times_{\Z}\F_p$ we also obtain a homological stability statement for \'etale cohomology with coefficients in $\Z_l$ with $l$ prime to $p$.

Whenever there is a homological stability result, the next step is to construct an object whose homology is given by the stable homology groups. Let us recall what is known in the topological case. We denote by $\Omega^2_0S^2$ the connected components of the constant map in the two-fold loop space $\Omega^2S^2$. 

\begin{theo*}[Segal, \cite{segalconfiguration}]\label{theo: Segal}
There is a map $C_n(\mathbf{R}^2)\to \Omega^2_0S^2$ which induces an isomorphism in homology in the stable range. 
\end{theo*}

In the motivic case, we have a similar result. We can consider the scheme $F$ which is the disjoint union of the schemes $F_n$ of degree $n$ pointed maps $\P^1\to\P^1$. As observed by Cazanave in \cite{cazanavealgebraic}, this object has the structure of a monoid such that the map $F\to\Z_{\geq 0}$ sending $F_n$ to $n$ is a monoid map. However, contrary to the topological case, $F$ is not grouplike (the computation of its $\pi_0$ is the main theorem of \cite{cazanavealgebraic}). We can form the homotopy colimit $F_\infty$ of the sequence
\[F_1\to F_2\to F_3\to\ldots\to F_n\to\ldots\]
where each map is obtained by multiplication with a fixed element in $F_1$. Our proof of motivic homological stability immediately implies the following theorem.

\begin{theo*}[\ref{prop: stable homology}]
The scanning map $C_n\to F_\infty$ induces an isomorphism in motivic cohomology in the stable range.
\end{theo*}

Moreover, the complex points of $F_\infty$ with the usual complex topology are a model for $\Omega_0^2S^2$ as shown in Proposition \ref{prop: Finfty=Omega0}.

\subsection*{Acknowledgments}

I wish to thank S{\o}ren Galatius, Martin Palmer and Federico Cantero for helpful conversations about homological stability, Cl\'ement Dupont for teaching me about mixed Tate motives, Martin Frankland for sharing his expertise on triangulated categories, Sam Nariman for pointing out the reference \cite{farbtopology}, Jesse Wolfson for finding a mistake in a former version of this paper and the anonymous referee for several useful comments.

\section{t-structures and homological stability}

Recall the following definition.

\begin{defi}\label{defi: t-structure}
Let $\cat{T}$ be a triangulated category. A \textbf{$t$-structure} on $\cat{T}$ is the data of two full subcategories $\cat{T}_{\geq 0}$ and $\cat{T}_{< 0}$ satisfying the following conditions:
\begin{itemize}
\item If $U\in\cat{T}_{\geq 0}$ and $V\in\cat{T}_{< 0}$, then $\cat{T}(U,V)=0$.
\item The category $\cat{T}_{\geq 0}$ is stable under suspensions and the category $\cat{T}_{< 0}$ is stable under desuspension.
\item There are functors $t_{< 0}$ and $t_{\geq 0}$  from $\cat{T}$ to $\cat{T}$ such that the image of $t_{<0}$  is contained in $\cat{T}_{<0}$ and the image of $t_{\geq 0}$ is contained in $\cat{T}_{\geq 0}$. Moreover, there is a natural distinguished triangle $t_{\geq 0} X\to X\to t_{<0}X$.
\end{itemize}
\end{defi}

\begin{rem}\label{rem: truncation functors}
Usually (e.g. in \cite[Definition IV.4.2. p. 278]{gelfandmethods}), the third axiom is stated by requiring that any object $X$ fits into a cofiber sequence
\[t_{\geq 0} X\to X\to t_{<0}X\]
with $t_{\geq 0}X\in\cat{T}_{\geq 0}$ and $t_{<0}X\in\cat{T}_{<0}$. However, our definition is not less general as it can be shown that this distinguished triangle can be made functorial and that the functor $t_{\geq 0}$ is right adjoint to the inclusion which proves that it is unique up to unique isomorphism and similarly for $t_{<0}$. A proof of these facts can be found in \cite[Lemma IV.4.5 p.279]{gelfandmethods}. In particular, all the cofiber sequences $A\to X\to B$ with $A$ in $\cat{T}_{\geq 0}$ and $B$ in $\cat{T}_{<0}$ are isomorphic.
\end{rem}

Let $\cat{T}$ be a triangulated category equipped with a $t$-structure. For $d$ an integer, we denote by $\cat{T}_{\geq d}$ the full subcategory  of $\cat{T}$ spanned by objects isomorphic to $X[d]$ with $X$ in $\cat{T}_{\geq 0}$. We denote by $\cat{T}_{<d}$ the full subcategory spanned by objects isomorphic to $X[d]$ with $X$ in $\cat{T}_{<0}$. We denote by $t_{\geq d}$ the functor $X\mapsto (t_{\geq 0} X[-d])[d]$ and $t_{<d}$ the functor $X\mapsto (t_{<0}X[-d])[d]$. We have a natural distinguished triangle
\[t_{\geq d}X\to X\to t_{<d}X\]

For $\cat{T}$ a triangulated category equipped with a $t$-structure, we denote by $\cat{T}_{\heartsuit}$ the full subcategory spanned by objects that are both in $\cat{T}_{\geq 0}$ and in $\cat{T}_{<1}$. The category $\cat{T}_{\heartsuit}$ is an abelian category (see for instance \cite[IV.4.7.]{gelfandmethods} for a proof).

\begin{defi}\label{defi: t-structure functor}
Let $\cat{T}$ and $\cat{U}$ be two triangulated categories equipped with $t$-structure. We say that a triangulated functor $B:\cat{T}\to\cat{U}$ is \textbf{compatible with the $t$-structure} if $B(\cat{T}_{\geq 0})\subset \cat{U}_{\geq 0}$ and $B(\cat{T}_{<0})\subset \cat{U}_{<0}$.
\end{defi}

Note that a triangulated functor that is compatible with the $t$-structure induces a functor $\cat{T}_\heartsuit\to\cat{U}_\heartsuit$. The resulting functor of abelian categories is exact. Since we were unable to find this result in the literature, we have included a proof.

\begin{prop}\label{prop: t-exact functor induces exact functor on the heart}
If $B:\cat{T}\to\cat{U}$ is a triangulated functor that is compatible with the $t$-structure, the induced functor $B:\cat{T}_\heartsuit\to\cat{U}_\heartsuit$ is exact.
\end{prop}

\begin{proof}
It suffices to prove that $B$ preserves kernels and cokernels. We do the case of kernels, the other case is similar. Given a map $f:X\to Y$ in $\cat{T}_{\heartsuit}$, it is shown in \cite[IV.4.7.]{gelfandmethods} that its kernel is obtained by the composite $t_{\geq 0}F\to F\to X$ where $F\to X\to Y\to F[1]$ is a distinguished triangle completing $f$ (note that the authors of \cite{gelfandmethods} use the cohomological grading as opposed to the homological grading used in our paper which explains the difference). Thus, we have an exact sequence
\[0\to t_{\geq 0}F\to X\xrightarrow{f} Y\]
The functor $B$ sends it to
\[0\to B(t_{\geq 0}F)\to B(X)\xrightarrow{B(f)} B(Y)\]
and we have to check that this is exact. In other words, we want to prove that $B(t_{\geq 0}F)$ is isomorphic to $t_{\geq 0}B(F)$. We know that there is a cofiber sequence
\[B(t_{\geq 0}F)\to B(F)\to B(t_{<0}F)\]
which by Remark \ref{rem: truncation functors} must be isomorphic to
\[t_{\geq 0}B(F)\to B(F)\to t_{<0}B(F)\]
\end{proof}

We now observe that a $t$-structure is all that is needed in order to be able to define the concept of homological stability.

\begin{defi}\label{defi: homological stability}
Let $l:\mathbf{N}\to\mathbf{N}$ be a function tending to $+\infty$. We say that a sequence of objects $\{M_d\}_{d\in\mathbf{N}}$ of a triangulated category $\cat{T}$ with $t$-structure $(\cat{T}_{\geq 0},\cat{T}_{<0})$ satisfies \textbf{homological stability with slope $l$} if, for each $d$, there exists an isomorphism $t_{< l(d)}M_d\xrightarrow{\cong} t_{< l(d)}M_{d+1}$.
\end{defi}

\begin{rem}
The derived category of a ring $\Lambda$, denoted $\D(\Lambda)$, has a standard $t$-structure in which the objects of $\D(\Lambda)_{\geq 0}$ are the complexes whose homology is concentrated in nonnegative degrees and the objects of $\D(\Lambda)_{<0}$ are the complexes whose homology is concentrated in negative degrees. Given a sequence of spaces $\{X_d\}_{d\in \N}$, we can consider the sequence $\{C_*(X_d,\Lambda)\}_{d\in\N}$ of objects of $\D(\Lambda)$. Then, if this sequence satisfies homological stability with slope $l$, for any $i<l(d)$, there is an isomorphism $\H_i(X_d,\Lambda)\cong \H_i(X_{d+1},\Lambda)$. Hence, classical homological stability is a particular case of our definition.
\end{rem}

\begin{rem}
In many cases, the sequence of objects $\{M_d\}$ is equipped with maps $f_d:M_d\to M_{d+1}$ and the isomorphisms $t_{< l(d)}M_d\xrightarrow{\cong} t_{< l(d)}M_{d+1}$ are given by $t_{< l(d)}(f_d)$. However there are cases where such maps are not available. This happens in the case of homological stability for configuration spaces in a closed manifold \cite{benderskylocalization,canterohomological} or in our main theorem.
\end{rem}

We have the following lemma whose proof is trivial.

\begin{lemm}\label{lemm: transfer of homological stability}
Let $\cat{T}$ be a triangulated category with a $t$-structure. Let $\{M_d\}_{d\in\mathbf{N}}$ and $\{N_d\}_{d\in\mathbf{N}}$ be two sequences of objects of $\cat{T}$ and let $j_d:M_d\to N_d$ be maps. Assume that there exists a function $m:\mathbf{N}\to\mathbf{N}$ tending to infinity such that all the maps $t_{<m(d)}(j_d)$ are isomorphisms. Then if one of $\{M_d\}$ or $\{N_d\}$ has homological stability with slope $l$, the other has homological stability with slope $\mathrm{min}(l,m)$.
\end{lemm}

\section{Mixed motives}\label{section: mixed motives}

We denote by $\MS(\Z)$ the model category of $\A^1$-spaces over $\on{Spec}(\Z)$. Its underlying category is the category of simplicial presheaves over the category $\cat{Sm}_{\Z}$ of smooth quasi-projective schemes over $\on{Spec}(\Z)$. The model structure is obtained by left Bousfield localization of the injective model structure by imposing descent with respect to the Nisnevich hypercovers and the fact that for any smooth schemes $U$, the map $U\times\A^1\to U$ is a weak equivalence. We can construct in an analogous way the model category $\MS(R)$ of $\A^1$-spaces over $R$ for any commutative ring $R$. The functor $j:X\mapsto X\times_{\on{Spec}(\Z)}\on{Spec}(\Q)$ is a morphism of sites from the Nisnevich site of smooth schemes over $\Z$ to the Nisnevich site of smooth schemes over $\Q$. It induces a symmetric monoidal left Quillen functor $j^*:\MS(\Z)\to\MS(\Q)$. We get similar functors $i_p^*:\MS(\Z)\to\MS(\F_p)$ for each prime number $p$.

We denote by $\DM(\Z)$ (resp. $\DM(\Q)$, resp. $\DM(\F_p)$) the triangulated category of motives with $\Z$-coefficients over $\on{Spec}(\Z)$ (resp. $\on{Spec}(\Q)$, resp $\on{Spec}(\F_p)$). There are several available models for such a category. We use the model constructed by Spitzweck in \cite{spitzweckcommutative}. One first constructs the model category of motivic spectra over $\Z$ (resp. over $\Q$, resp. over $\F_p$). This is obtained from the model category of pointed objects of $\MS(\Z)$ (resp. $\MS(\Z)$, resp. $\MS(\Z)$) by forcing the motivic space $\P^1$ with base point $\infty$ to be invertible with respect to the smash product. The category $\DM(\Z)$ is then the homotopy category of modules over an $\mathscr{E}_\infty$-algebra $\MZ$ (resp. $\MZ_{\Q}$, resp. $\MZ_{\F_p}$) in the symmetric monoidal model category of motivic spectra over $\Z$ (resp. over $\Q$, resp. over $\F_p$) called the motivic Eilenberg-MacLane spectrum.  The functor $j^*$ from $\MS(\Z)$ to $\MS(\Q)$ induces a symmetric monoidal left adjoint $j^*:\DM(\Z)\to \DM(\Q)$. Similarly, for each prime $p$, the map $i_p:\Z\to\F_p$ induces a symmetric monoidal left adjoint $i_p^*:\DM(\Z)\to\DM(\F_p)$.

As proved in \cite[Theorem 1]{rondigsmodules}, the triangulated category $\DM(\Q)$ is equivalent to Voevodsky's big category of motives over $\on{Spec}(\Q)$ as a symmetric monoidal triangulated category and similarly, the triangulated category $\DM(\F_p)$ is equivalent to Voevodsky's big category of motives over $\on{Spec}(\F_p)$.

When seen as an object of $\DM(\Z)$, the motivic Eilenberg-MacLane spectrum $\MZ$ is simply denoted $\Z(0)$. More generally, for $q$ a nonnegative integer, we denote by $\Z(q)$ the $\MZ$-module $\MZ\wedge(\Sigma^\infty\P^1)^{\wedge q}[-2q]$. For a negative integer $q$, we denote by $\Z(q)$ the object $\Hom_{\DM(\Z)}(\Z(-q),\Z(0))$ where $\Hom_{\DM(\Z)}$ denotes the inner Hom in the symmetric monoidal category $\DM(\Z)$. Note that the triangulated category $\DM(\Z)$ is tensored over $\D(\Z)$, the derived category of chain complexes over $\Z$. For $A$ an object of $\D(\Z)$ and $q$ an integer, we denote by $A(q)$ the tensor product $A\otimes\Z(q)$.

We define similarly objects $\Z_\Q(q)$ in the triangulated category $\DM(\Q)$ and $\Z_{\F_p}(q)$ in $\DM(\F_p)$. We have an isomorphism $j^*\Z(q)\cong \Z_\Q(q)$ (see \cite[Theorem 7.6. and Theorem 7.18.]{spitzweckcommutative}) and $i_p^*\Z(q)\cong \Z_{\F_p}(q)$ (see \cite[Theorem 9.16.]{spitzweckcommutative}).

The triangulated category $\DM(\Z)$ comes equipped with a symmetric monoidal functor $\M:\cat{Sm}_\Z\to\DM(\Z)$ which sends to a smooth scheme its motive. For instance, we have the classical formula
\[\M(\P^n)\cong\Z(0)\oplus\Z(1)[2]\oplus\ldots\oplus\Z(n)[2n]\]

\begin{defi}\label{defi: motivic cohomology}
For $X$ a scheme over $\Z$, we define its \textbf{motivic cohomology} of degree $p$ with coefficients in $\Z(q)$ by the following formula:
\[\H^p(X,\Z(q)):=\DM(\Z)(\M(X),\Z(q)[p])\]
\end{defi}

Spitzweck proves in \cite[Corollary 7.19.]{spitzweckcommutative} that this definition coincides with the definition of motivic cohomology of schemes over $\Z$ given by Levine in \cite{levinektheory} by mean of a generalization of Bloch cycle complexes. 

If $X$ is a scheme over $\Q$, we denote by $\H^p(X,\Z(q))$ the motivic cohomology groups given by the formula:
\[\H^p(X,\Z(q))=\DM(\Q)(\M(X),\Z_\Q(q)[p])\]
and similarly for schemes over $\F_p$. These motivic cohomology groups coincide with the motivic cohomology groups defined in \cite{mazzalecture}.

We will need the following fact.

\begin{theo}\label{prop: reduction to fields}
For all integers $n$ and $r$, there is an exact sequence
\[\prod_p \H^{n-2}(\on{Spec}(\F_p),\Z(r-1))\to\H^n(\on{Spec}(\Z),\Z(r))\to\H^n(\on{Spec}(\Q),\Z(r))\]
where the product is taken over the set of all prime numbers.
\end{theo}

\begin{proof}
This follows from the distinguished triangle in \cite[Lemma 7.8.]{spitzweckcommutative} by applying derived global sections.
\end{proof}

\subsection*{Gysin triangles}

The following theorem is one of the most useful tool for computations in the category $\DM(\Z)$.

\begin{theo}\label{theo: Gysin}
Let $X$ be a smooth scheme and $Z$ be a closed smooth subscheme of codimension $n$. Let $U$ be the complement of $Z$ in $X$. Then there is a cofiber sequence in $\DM(\Z)$
\[\M(U)\to\M(X)\to \M(Z)(n)[2n]\]
where the first map is induced by the inclusion $U\to X$.
\end{theo}

\begin{proof}
See \cite[Definition 4.6.]{deglisearound}.
\end{proof}

\subsection*{Betti realization}

We denote by $\DM_{gm}(\Z)$ the thick subcategory spanned by the motives $\M(X)$ for $X$ any smooth scheme. According to \cite[Theorem 11.1.13]{cisinskitriangulated} the objects of $\DM_{gm}(\Z)$ are exactly the compact objects of $\DM(\Z)$. The category $\DM_{gm}(\Z)$ inherits the symmetric monoidal structure of $\DM(\Z)$ and contains all the motives $\Z(n)$ with $n\in\Z$. We define analogously a thick subcategory $\DM_{gm}(\Q)$ of $\DM(\Q)$. Note that there is an isomorphism
\[j^*\M(X)\cong \M(X\times_{\on{Spec}(\Z)}\on{Spec}(\Q))\]
It follows that the functor $j^*$ restricts to a functor
\[j^*:\DM_{gm}(\Z)\to\DM_{gm}(\Q)\]

There is a Betti realization functor $\DM_{gm}(\Q)\to \D(\Z)$ constructed in \cite[Section 3.3.]{lecomterealisation}. We denote by $B$ its precomposition with the map $j^*$ and call it the Betti realization of a motive on $\on{Spec}(\Z)$. This is a symmetric monoidal triangulated functor. Its restriction to smooth schemes is understood via the following proposition.

\begin{prop}\label{prop: Betti realization on smooth schemes}
Let $X$ be a smooth scheme over $\Z$, then, there is an isomorphism
\[\H_q(B(\M(X)))\cong \H_q(X(\C),\Z)\]
where the right hand side denotes the singular homology groups of $X(\C)$ seen as a complex manifold.
\end{prop}

We will also need the following fact. 

\begin{prop}\label{prop: Betti realization of Tate twists}
Let $A$ be a bounded chain complex of finitely generated abelian groups. Then there is an isomorphism $B(A(q))\cong A$ in the category $\D(\Z)$.
\end{prop}

\begin{proof}
Let $\D_c(\Z)$ be the full subcategory of $\D(\Z)$ spanned by bounded chain complexes of finitely generated abelian groups. Let $c_q:\D_c(\Z)\to \DM_{gm}(\Z)$ be the functor sending $A$ to $A(q)$. We want to prove that $B\circ c_q$ is isomorphic to the identity. Since both $B\circ c_q$ and $\id$ are triangulated, it suffices to show that they coincide on $\Z$. In other words, we are reduced to proving that $B(\Z(q))\cong \Z$. Since the functor $B$ is symmetric monoidal, the result is true for $q=0$ and it suffices to prove that $B(\Z(1))\cong \Z$. We have an isomorphism $\M(\P^1)\cong \Z(0)\oplus \Z(1)[2]$. Hence we have
\[\Z\oplus \Z[2]\cong B(\M(\P^1))\cong B(\Z)\oplus B(\Z(1))[2]\]
which  gives the desired answer.
\end{proof}

\section{Mixed Tate motives}

We denote by $\DMT(\Z)$ the full triangulated subcategory of $\DM(\Z)$ generated by the motives $\Z(q)$ for all integers $q$. We have the following identities in the category $\DMT(\Z)$:
\begin{align*}
\DMT(\Z)(\Z(i)[m],\Z(j)[n])&=0,\;\mathrm{if}\;j<i,\\
\DMT(\Z)(\Z(i)[m],\Z(i)[n])&=0,\;\mathrm{if}\;m\neq n,\\
\DMT(\Z)(\Z(i)[n],\Z(i)[n])&=\Z.
\end{align*}
They can be deduced from Theorem \ref{prop: reduction to fields} and the fact that similar results hold in $\DM(\Q)$ and $\DM(\F_p)$. It follows that the triangulated category $\DMT(\Z)$ satisfies the axioms of a triangulated category of Tate type \cite[Definition 1.1.]{levinetate}. The first consequence (\emph{cf.} \cite[Lemma 1.2.]{levinetate}) is that for any object $M$ in $\DMT(\Z)$, there is a functorial weight filtration
\[\ldots\to W_{\leq n}M\to W_{\leq n+1}M\to\ldots\to M\]
where $W_{\leq n}$ is the right adjoint to the inclusion functor $i_n$ from the full triangulated subcategory of $\DMT(\Z)$ generated by the objects $\Z(-q)$ with $q\leq n$. These maps fit into cofiber sequences
\[W_{\leq n-1}M\to W_{\leq n}M\to \gr_n^WM\]
in which $\gr_n^WM$ is of the form $A(-n)$ for some bounded chain complex of finitely generated $\Z$-module $A$. In fact, as observed by Kahn in \cite{kahnweight}, the functor $A\mapsto A(-n)$ is an equivalence from the triangulated category of bounded chain complexes of finitely generated $\Z$-modules to the full triangulated subcategory of $\DMT(\Z)$ generated by $\Z(-n)$. Note that Levine works with $\Q$-coefficients, however a careful inspection of his proof shows that it also works with $\Z$-coefficients. 

\begin{lemm}\label{lemm: boundedness of the weight filtration}
For any object $M$ of $\DMT(\Z)$, the weight filtration is bounded, i.e. $W_{\leq n}M=0$ for $n$ small enough and $W_{\leq n}M\to M$ is an isomorphism for $n$ big enough. 
\end{lemm}

\begin{proof}
Since the functors $W_{\leq n}$ are triangulated functors, the class of objects whose weight filtration is bounded is a triangulated subcategory of $\DMT(\Z)$, moreover it contains all the generators, hence it has to be the whole category.
\end{proof}

\begin{theo}
The category $\DMT(\Z)$ satisfies the Beilinson--Soul\'e vanishing conjecture. That is, we have the identity
\[\DMT(\Z)(\Z(0)[0],\Z(q)[n])=0\]
whenever $q\geq 0$ and $n<0$ and whenever $q>0$ and $n\leq 0$.
\end{theo}

\begin{proof}
This statement also appears as \cite[Theorem 2.10.]{souderesmotivic}. According to Theorem \ref{prop: reduction to fields}, it suffices to prove that $\DMT(\Q)$ and all the categories $\DMT(\F_p)$ satisfy the Beilinson--Soul\'e vanishing conjecture with $\Z$ coefficients. As explained in \cite[Lemma 24]{kahnalgebraic}, we are reduced to proving that, for $k=\Q$ and $k=\F_p$, the category $\DMT(k)_{\Z/l}$ satisfies the Beilinson--Soul\'e vanishing conjecture for all prime number $l$ and that $\DMT(k)_{\Q}$ satisfies the Beilinson--Soul\'e vanishing conjecture. The case of $\Z/l$ with $l$ prime to the exponent characteristic of $k$ follows from the Beilinson--Lichtenbaum conjecture which has now been proved by Voevosdky in \cite{voevodskymotivic}. The case $\DMT(\F_p)_{\Z/p}$ follows from \cite[Theorem 1.1.]{geisserktheory}.

For $k$ a field, we have the identity
\[\DMT(k)_{\Q}(\Q(0)[0],\Q(i)[2j-i])\cong (K_i(k)\otimes\Q)^{(j)}\]
where the right hand side denotes the weight $j$ eigenspace of the Adams operations action on algebraic $K$-theory. For finite fields, work of Quillen has shown that $K_*(k)\otimes \Q$ is trivial which implies the vanishing conjecture. For $k=\Q$ or more generally a number field, Borel has explicitly computed $K_*(k)\otimes\Q$ and it follows from his computation that the vanishing conjecture is also true in this case (see \cite{kahnalgebraic} for references).
\end{proof}

We denote by $\DMT(\Z)_{\geq 0}$ (resp. $\DMT(\Z)_{<0}$) the full subcategory of $\DMT(\Z)$ spanned by objects $M$ such that for each integer $n$, the motive $\gr^W_n(M)$ has homology concentrated in nonnegative (resp. negative) degree.

\begin{theo}[Levine]
The pair $(\DMT(\Z)_{\geq 0},\DMT(\Z)_{<0})$ is a $t$-structure on $\DMT(\Z)$.
\end{theo}

\begin{proof}
See \cite[Theorem 1.4.]{levinetate}. Again, Levine restricts to $\Q$ coefficients but this restriction plays no role in the proof.
\end{proof}

Recall that we have a Betti realization functor $B:\DM_{gm}(\Z)\to \D(\Z)$, we still denote by $B$ its restriction to $\DMT(\Z)\subset\DM_{gm}(\Z)$.

\begin{prop}\label{prop: B compatible with t-structure}
The functor $B$ is compatible with the $t$-structure (Definition \ref{defi: t-structure functor}).
\end{prop}

\begin{proof}
We prove that $B(\DMT(\Z)_{\geq 0})\subset \D(\Z)_{\geq 0}$, the other case is similar. Let $M$ be a motive in $\DMT(\Z)$, then according to Lemma \ref{lemm: boundedness of the weight filtration}, $M$ has a finite filtration.
\[0=W_{\leq n}(M)\to W_{\leq n+1}(M)\to\ldots \to W_{\leq n+m}M=M\]
Assume that $M$ is in $\DMT(\Z)_{\geq 0}$. We show by induction on $q$ that $B(W_{\leq n+q}M)$ in in $\D(\Z)_{\geq 0}$. For $q=1$, this follows from Proposition \ref{prop: Betti realization of Tate twists} and the fact that $W_{\leq n+1}M$ is isomorphic to $A(-n-1)$ with $A$ a chain complex with zero homology in negative degree. Assume that this is true for $q$. We have a cofiber sequence
\[W_{\leq n+q}M\to W_{\leq n+q+1}M\to \gr^W_{n+q+1}M\]
with $\gr^W_{n+q+1}M$ isomorphic to $A(-n-q-1)$ with $A$ a chain complex whose homology is zero in negative degree. Applying $B$ to this cofiber sequence, we find that $B(W_{\leq n+q+1}M)$ is in $\D(\Z)_{\geq 0}$ as desired.  
\end{proof}

\begin{prop}\label{prop: B conservative}
The functor $B$ reflects the zero object.
\end{prop}

\begin{proof}
The previous proposition implies that $B$ sends $\DMT(\Z)_{\heartsuit}$ to $\D(\Z)_{\heartsuit}$. We denote by $B_\heartsuit$ the restriction of $B$ to $\DMT(\Z)_{\heartsuit}$. By Proposition \ref{prop: t-exact functor induces exact functor on the heart}, the functor $B_{\heartsuit}$ is an exact functor between abelian categories. We first prove that $B_\heartsuit$ reflects the zero object. Let $X$ be an object of $\DMT(\Z)_\heartsuit$ that is sent to $0$ by $B_\heartsuit$. According to \cite[Theorem 1.4.]{levinetate}, the weight filtration for $X$ has the form
\[0=W_{\leq n}X\subset W_{\leq n+1}X\subset\ldots\subset W_{\leq n+q}X=X\] 
In particular, we have a short exact sequence
\[0\to B_\heartsuit(W_{\leq n+q-1}X)\to B_\heartsuit(W_{\leq n+q}X)\to B_\heartsuit(\gr^W_{n+q}X)\to 0\]
Since we have assumed that $B_\heartsuit(W_{\leq n+q}X)=0$, we find that $B_\heartsuit(W_{\leq n+q-1}X)=0$ and $B_\heartsuit(\gr^W_{n+q}X)=0$. The object $\gr^W_{n+q}X$ is isomorphic to $A(-n-q)$ with $A$ a finitely generated abelian group. Therefore, by Proposition \ref{prop: Betti realization of Tate twists}, $B_\heartsuit(\gr^W_{n+q}X)\cong A$. Since $B_\heartsuit(\gr^W_{n+q}X)$ is trivial, we find that $\gr^W_{n+q}(X)=0$. An easy inductive argument shows that $B_\heartsuit(X)=0$.

Now we prove that $B$ reflects the zero object. Let $X\in\DMT(\Z)$ be such that $B(X)=0$. Assume that the map $t_{\geq n}X\to X$ is an isomorphism. Then, there is a cofiber sequence
\[t_{\geq n+1}X\to t_{\geq n} X\to Y[n]\]
where $Y$ is in $\DMT(\Z)_\heartsuit$. Applying $B$ to this cofiber sequence, we find a cofiber sequence
\[B(t_{\geq n+1} X)\to B(t_{\geq n}X)\to B_\heartsuit(Y)[n]\]
Looking at the induced exact sequence in degree $n$ and $n-1$, we find that $B_\heartsuit(Y)$ is $0$ which implies by the previous paragraph that $Y=0$. Hence, we have proved that $t_{\geq n+1}X\to X$ is an isomorphism. An obvious induction then shows that $X$ is in $\bigcap_{n \in\Z}\DMT(\Z)_{\geq n}$ and hence is $0$ by non-degeneracy of the $t$-structure \cite[Theorem 1.4.i]{levinetate}.
\end{proof}

\begin{coro}
Same notations, a mixed Tate motive is in $\DMT(\Z)_{\geq 0}$ if and only if $B(X)$ is in $\D(\Z)_{\geq 0}$.
\end{coro}

\begin{proof}
One direction is given by Proposition \ref{prop: B compatible with t-structure}. Now assume that $B(X)\in\D(\Z)_{\geq 0}$. We have a cofiber sequence
\[t_{\geq 0}X\to X\to t_{<0}X\]
Applying $B$, we find a cofiber sequence
\[B(t_{\geq 0} X)\to B(X)\to B(t_{<0}X)\]
From this cofiber sequence, we deduce that $B(t_{<0}X)$ is in $\D(\Z)_{\geq 0}$ and from Proposition \ref{prop: B compatible with t-structure}, we find that $B(t_{<0}X)$ is in $\D(\Z)_{<0}$. It follows that $B(t_{<0}X)=0$ and by Proposition \ref{prop: B conservative}, we deduce that $t_{<0}X=0$.
\end{proof}

\begin{coro}\label{coro: B detects the t-structure}
Let $f:M\to N$ be a map in $\DMT(\Z)$. Then $t_{<n}(f)$ is an isomorphism if and only if $t_{<n}B(f)$ is an isomorphism.
\end{coro}

\begin{proof}
Let us assume that $t_{<n}(f)$ is an isomorphism. Let us complete $f$ in a cofiber sequence
\[M\xrightarrow{f} N\to Q\]
Hitting this cofiber sequence with the triangulated functor $\gr_q^W$ for all $q$, we find that $t_{<n}(f)$ is an isomorphism if and only if $\gr^W_qQ$ has homology concentrated in degree $\geq n$ for all $q$. In other words, $t_{<n}(f)$ is an isomorphism if and only if $Q$ is in $\DMT(\Z)_{\geq n}$. By the previous corollary, this happens if and only if $B(Q)$ is in $\D(\Z)_{\geq n}$. Writing the long exact sequence associated to the cofiber sequence
\[B(M)\xrightarrow{B(f)}B(N)\to B(Q),\]
we find that this happens if and only if $t_{<n}B(f)$ is an isomorphism.
\end{proof}

\section{Configuration spaces and the scanning map}\label{section: scanning}

Given two monic polynomials $f=x^d+a_{d-1}x^{d-1}+\ldots+a_0$ and $g=x^d+b_{d-1}x^{d-1}+\ldots+b_0$ with coefficients in a commutative ring $R$, one can form their resultant $\on{Res}(f,g)$. A definition can be found in \cite[IV,8]{langalgebra}. The important property of this construction is that, if $R$ is a field, the resultant is zero if and only if the two polynomials have a common root in some extension of $R$ (\emph{cf.} \cite[IV, Corollary 8.4.]{langalgebra}).

We can then define $F_d$ as the Zariski open subset of $\A^{2d}=\mathrm{Spec}(\Z[a_0,\ldots,a_{d-1},b_0,\ldots,b_{d-1}])$ complement of the hypersurface of equation
\[\mathrm{Res}(x^d+a_{d-1}x^{d-1}+\ldots+a_0,x^d+b_{d-1}x^{d-1}+\ldots+b_0)=0\]
Thus, for $k$ a field, the $k$-points $F_d(k)$ is the set of pairs of monic polynomials $(f,g)$ with no common factor. We think of $F_d$ as the scheme of degree $d$ maps from $\P^1$ to $\P^1$ sending $\infty$ to $1$. A point $(f,g)$ of $F_d$ represents the map $x\mapsto \frac{f(x)}{g(x)}$ from $\P^1$ to $\P^1$.

We denote by $\Map_*(\P^1(\C), \P^1(\C))$ the topological space of maps from $\P^1(\C)$ to $\P^1(\C)$ sending the point at $\infty$ to $1$. We have an isomorphism
\[\pi_0[\Map_*(\P^1(\C),\P^1(\C))]\cong \pi_0(\Omega^2S^2)\cong \pi_2(S^2)\cong\Z\]
This isomorphism is implemented by the degree. We denote by $\Map_*(\P^1(\C),\P^1(\C))_d$ the component corresponding to the integer $d$. Denoting $F_d(\C)$ the set of complex points of $F_d$ with its complex manifold structure, there is a continuous inclusion
\[F_d(\C)\to\Map_*(\P^1(\C),\P^1(\C))_d\]
which sends $(f,g)$ to the meromorphic function $\frac{f}{g}$.

There is a map $s_d:\A^d\to \A^{2d}$ given by the formula
\[s_d(a_0,\ldots,a_{d-1})=(a_0,\ldots,a_{d-1},a_0+a_1,a_1+2a_2,\ldots,a_{d-2}+(d-1)a_{d-1},a_{d-1}+d)\]
If we think as $\A^d(R)$ as the set of monic polynomials over $R$ of degree $d$ and as $\A^{2d}(R)$ as the set of pairs of monic polynomials over $R$ of degree $d$, the map $s_d$ sends the polynomial $f$ to $(f,f+f')$.

We define the scheme $C_d$ by the following pullback square in schemes
\[
\xymatrix{
C_d\ar[d]\ar[r]&F_d\ar[d]\\
\A^d\ar[r]_{s_d}&\A^{2d}
}
\]
where the map $F_d\to\A^{2d}$ is the obvious inclusion. For $k$ a field, the set $C_d(k)$ is the set of monic degree $d$ polynomials over $k$ such that $f$ and $f+f'$ are coprime. Equivalently, this is the set of polynomials $f$ that are coprime with their derivative. In particular, if $k$ is an algebraically closed field, $C_d(k)$ gets identified with the set of unordered configurations of $d$ distinct points in $k$ by sending a polynomial $f$ to its set of roots. If $k=\C$, the complex manifold $C_d(\C)$ is exactly the complex manifold of unordered configurations of points in the complex line.

By definition, the scheme $C_d$ comes with a map $C_d\to F_d$. We still denote this map by $s_d$ as this is just the restriction of $s_d$ to an open subset of $\A^d$. We call it the degree $d$ scanning map. This is justified by the observation that the composite
\[C_d(\C)\xrightarrow{s_d(\C)}F_d(\C)\to\Map_*(\P^1(\C),\P^1(\C)_d\]
is homotopic to the topological scanning map described by Segal in \cite[Section 3]{segalconfiguration}.

\section{Proof of the main theorem}

Our first job is to prove that the motives $\M(C_d)$ and $\M(F_d)$ are in $\DMT(\Z)$. We will rely on the following proposition.

\begin{prop}\label{prop: filtration}
Let $\varnothing=Z_0\subset Z_1\subset \ldots\subset Z_n\subset X$ be a stratification of a smooth scheme $X$ by closed subschemes. Assume that $\M(X)$ is mixed Tate, that for all $k$, the scheme $Z_{k}-Z_{k-1}$ is smooth and that the motive $\M(Z_k-Z_{k-1})$ is mixed Tate. Then $\M(X-Z_n)$ is mixed Tate.
\end{prop}

\begin{proof}
We prove by induction on $k$ that $\M(X-Z_k)$ is mixed Tate. The case $k=0$ follows from the hypothesis. Assume  that $\M(X-Z_k)$ is mixed Tate. The scheme $Z_{k+1}-Z_k$ is closed in $X-Z_k$, therefore, applying the Gysin triangle \ref{theo: Gysin} to this pair, we get a cofiber sequence
\[\M(X-Z_{k+1})\to\M(X-Z_k)\to \M(Z_{k+1}-Z_k)(n_k)[2n_k]\]
where $n_k$ denotes the codimension of $Z_{k+1}-Z_k$ in $X-Z_k$. It follows that $\M(X-Z_{k+1})$ is mixed Tate as desired.
\end{proof}

Now, we use the paper \cite{farbtopology}. The authors of this paper study the family of varieties $Poly_n^{d,m}$. They observe in the paragraph following \cite[Definition 1.1.]{farbtopology} that $Poly_1^{d,2}$ is isomorphic to $F_d$ and $Poly_2^{d,1}$ is isomorphic to $C_d$. 

\begin{lemm}\label{C and F are mixed Tate}
The motives $\M(Poly_n^{d,m})$ are mixed Tate motives. In particular, for all $d$, $\M(C_d)$ and $\M(F_d)$ are mixed Tate motives.
\end{lemm}

\begin{proof}
We proceed by induction on $d$. This is evident for $d=0$. Let $d\geq 1$. When $d\leq n$, the assertion follows from \cite[Equation 3.1.]{farbtopology}. In the other cases, as explained in \cite[Equation 3.2.]{farbtopology}, we have a descending filtration
\[\A^{dm}=R_{n,0}^{d,m}\supset R_{n,1}^{d,m}\supset R_{n,2}^{d,m}\supset\ldots\supset \varnothing\]
where each $R_{n,k}^{d,m}$ is a closed subscheme of $\A^{dm}$. By \cite[Equation 3.3.]{farbtopology}, we have an isomorphism
\[R_{n,k}^{d,m}-R_{n,k+1}^{d,m}\cong Poly_n^{d-kn,m}\times\A^k\]
Using our induction hypothesis and Proposition \ref{prop: filtration}, we find that $\M(Poly_n^{d,m})\cong \M(\A^{dm}-R_{n,1}^{d,m})$ is mixed Tate.
\end{proof}

As explained in \cite[Proposition 3.1.]{cazanavealgebraic}, the scheme $\bigsqcup_dF_d$ is equipped with a graded monoid structure that we will denote $\oplus$. Let us pick a $\Z$-point $u$ in $F_1$, we get a map $\alpha: F_d\to F_{d+1}$ as the composite
\[F_d\times\on{Spec}(\Z)\xrightarrow{\id\times u} F_d\times F_1\xrightarrow{-\oplus-} F_{d+1}\]

Similarly, identifying $\P^1(\C)$ with $S^2$, a choice of pinch map $\P^1(\C)\to \P^1(\C)\vee \P^1(\C)$ induces a multiplication
\[\Map_*(\P^1(\C),\P^1(\C))\times\Map_*(\P^1(\C),\P^1(\C))\to \Map_*(\P^1(\C),\P^1(\C))\]
which is additive with respect to the degree. The point $u$ in $F_1$ induces a point $v$ in $\Map_*(\P^1(\C),\P^1(\C))_1$ and we get a map
\[\beta:\Map_*(\P^1(\C),\P^1(\C))_d\to \Map_*(\P^1(\C),\P^1(\C))_{d+1}\]
constructed as before.

\begin{prop}\label{prop: commutative square up to homotopy}
The diagram
\[\xymatrix{
F_d(\C)\ar[d]\ar[r]^{\alpha}&F_{d+1}(\C)\ar[d]\\
\Map_*(\P^1(\C),\P^1(\C))_d\ar[r]_{\beta}&\Map_*(\P^1(\C),\P^1(\C))_{d+1}
}
\]
commutes in the homotopy category of spaces.
\end{prop}

\begin{proof}
The scheme over $\C$, $M:=(\bigsqcup_dF_d)\times\on{Spec}(\C)$ can be seen as a graded monoid in the category $\MS(\C)$ of motivic spaces over $\C$. Let $\tilde{\P}^1$ be a fibrant replacement of $\P^1$ in the model category $\MS(\C)_*$ of pointed objects of $\MS(\C)$. We can form the object
\[\Omega^{\P^1}(\P^1):=\underline{\Map}_*(\P^1,\tilde{\P}^1 )\]
where $\underline{\Map}_*$ denotes the inner Hom in the category $\MS(\C)_*$. The object $\Omega^{\P^1}(\P^1)$ splits as a disjoint union of components indexed by the degree. We denote by $\Omega_d^{\P^1}(\P^1)$ the component of degree $d$. There is a pinch map $\P^1\to \P^1\vee \P^1$ in $\MS(\Z)$ (see \cite[Definition A.2.]{cazanavealgebraic}). It follows that $\Omega^{\P^1}(\P^1)$ has a binary multiplication. Moreover, this multiplication is additive with respect to the degree, we can thus form a map $\gamma:\Omega_d^{\P^1}(\P^1)\to\Omega^{\P^1}_{d+1}(\P^1)$ as the composite
\[\Omega_d^{\P^1}(\P^1)\times\ast\xrightarrow{\id\times \ast}\Omega_d^{\P^1}(\P^1)\times\Omega_1^{\P^1}(\P^1)\to\Omega^{\P^1}_{d+1}(\P^1)\]
where the base point $\ast$ in $\Omega_1^{\P^1}(\P^1)$ is induced from the chosen point in $F_1$ via the composite
\[F_1\to\underline{\Map}_*(\P^1,\P^1)\to\underline{\Map}_*(\P^1,\tilde{\P}^1)\]

Using \cite[Remark A.7.]{cazanavealgebraic} and the fact that taking complex points is a weak equivalence preserving functor from $\MS(\Z)$ to topological spaces, we find that the diagram
\[
\xymatrix{
F_d(\C)\ar[d]\ar[r]^\alpha&F_{d+1}(\C)\ar[d]\\
\Omega^{\P^1}_d(\P^1)(\C)\ar[r]_\gamma&\Omega^{\P^1}_{d+1}(\P^1)(\C)
}
\]
commutes in the homotopy category of spaces. 

By adjunction, we have a map of topological spaces
\[j:\Omega^{\P^1}(\P^1)(\C)\to \Map_*(\P^1(\C),\tilde{\P}^1(\C))\]
If we give the right hand side, the multiplication coming from the pinch map on $\P^1(\C)$ obtained by taking the complex points of the pinch map $\P^1\to \P^1\vee \P^1$, the map $j$ commutes with multiplication on the nose. We thus get a commutative square
\[\xymatrix{
\Omega^{\P^1}_d(\P^1)(\C)\ar[r]^\gamma\ar[d]_j&\Omega^{\P^1}_{d+1}(\P^1)(\C)\ar[d]^j\\
\Map_*(\P^1(\C),\tilde{\P}^1(\C))_d\ar[r]_{\beta'}&\Map_*(\P^1(\C),\tilde{\P}^1(\C))_{d+1}
}
\]
where $\beta'$ is defined as the composite
\[\id\times v:\Map_*(\P^1(\C),\tilde{\P}^1(\C))_d\to\Map_*(\P^1(\C),\tilde{\P}^1(\C))_d\times \Map_*(\P^1(\C),\tilde{\P}^1(\C))_1\] 
and the multiplication map
\[\Map_*(\P^1(\C),\tilde{\P}^1(\C))_d\times \Map_*(\P^1(\C),\tilde{\P}^1(\C))_1\to\Map_*(\P^1(\C),\tilde{\P}^1(\C))_{d+1}\]

The motivic space $\tilde{\P}^1$ is weakly equivalent to $\P^1$ which implies that their complex points are weakly equivalent as topological spaces. Moreover, the motivic pinch map $\P^1\to\P^1\vee\P^1$ induces a map which is homotopic to the standard pinch map $\P^1(\C)\to\P^1(\C)\vee\P^1(\C)$. It follows that the map $\beta'$ represents the same map as the map $\beta$ in the homotopy category of spaces. Therefore, we get the desired result by gluing together the above two commutative squares.
\end{proof}

\begin{prop}\label{prop: stability for F}
The maps $\alpha:F_d\to F_{d+1}$ induce isomorphisms
\[t_{<d}\M(\alpha):t_{<d}\M(F_d)\to t_{<d}\M(F_{d+1})\]
\end{prop}

\begin{proof}
By Corollary \ref{coro: B detects the t-structure}, it suffices to prove that $B\M(\alpha)$ is an isomorphism in homology in degree smaller than $d$. By Proposition \ref{prop: commutative square up to homotopy}, the following square commutes up to homotopy.
\[
\xymatrix{
F_d(\C)\ar[r]^{\alpha}\ar[d]_{i_d}&F_{d+1}(\C)\ar[d]^{i_{d+1}}\\
\Map_*(\P^1(\C),\P^1(\C))_d\ar[r]_{\beta}&\Map_*(\P^1(\C),\P^1(\C))_{d+1}
}
\]
where the maps $i_d$ and $i_{d+1}$ are the obvious inclusions. According to \cite[Proposition 1.1.]{segaltopology}, the maps $i_d$ and $i_{d+1}$ induce isomorphisms in homology up to degree $d$, on the other hand, the map $\beta$ is a homotopy equivalence since $\Map_*(\P^1(\C),\P^1(\C))$ is a loop space. It follows that the map $B\M(\alpha)$ induces an isomorphism in homology in degree smaller than $d$.
\end{proof}

We are now ready to prove our main theorem. We denote by $l(d)$ the function $\on{min}(d,\lfloor{d/2}\rfloor+2)$. Note that for $d\geq 3$ this coincides with $\lfloor{d/2}\rfloor+2$.

\begin{theo}\label{theo: main bis}
The sequence of motives $\M(C_d)$ has homological stability with slope $l(d)$.
\end{theo}

\begin{proof}
By Lemma \ref{lemm: transfer of homological stability}, it suffices to prove that the scanning maps $s_d:C_d\to F_d$ induce an isomorphism
\[t_{<l(d)}\M(s_d):t_{<l(d)}\M(C_d)\to t_{<l(d)}\M(F_d)\]
By Corollary \ref{coro: B detects the t-structure}, we are reduced to proving that $B\M(s_d):B(\M(C_d))\to B(\M(F_d))$ induces an isomorphism in homology in degree $< l(d)$. By Proposition \ref{prop: Betti realization on smooth schemes}, it suffices to prove that the map
\[\H_*(C_d(\C),\Z)\to\H_*(F_d(\C),\Z)\]
induce an isomorphism for $*<l(d)$. We have a sequence of maps
\[H_*(C_d(\C),\Z)\to H_*(F_d(\C),\Z)\to H_*(\Map_*(\P^1(\C),\P^1(\C))_d,\Z)\]
The composite of these two maps is an isomorphism in homology in degree $<l(d)$ because we are in the stable range for $\H_*(C_d(\C),\Z)$. The second map induce an isomorphism for $*<l(d)$ (and even for $*<d$) by \cite[Theorem 3]{segalconfiguration}.
\end{proof}

Theorem \ref{theo: main} now follows easily.

\begin{coro}\label{coro: proof of main}
There is an isomorphism
\[\H^i(C_d,\Z(q))\cong\H^i(C_{d+1},\Z(q)) \]
for any $q$ and $i<l(d)$.
\end{coro}

\begin{proof}
We have a zig-zag of maps
\[\M(C_d)\xrightarrow{\M(s_d)}\M(F_d)\xrightarrow{\M(\alpha)} \M(F_{d+1})\xleftarrow{\M(s_{d+1})}\M(C_{d+1})\]

After hitting this zig-zag with $t_{<l(d)}$, all the maps become isomorphisms by the proof of the previous theorem. As explained in Remark \ref{rem: truncation functors}, the functor $t_{<l(d)}$ is the left adjoint to the inclusion $\DMT(\Z)_{<l(d)}\to\DMT(\Z)$. Moreover for $i<l(d)$, the object  $\Z(q)[i]$ is in $\DMT(\Z)_{<l(d)}$. It follows that each of the map in the above zig-zag induce an isomorphism when we apply the functor $\DMT(\Z)(-,\Z(q)[i])$.
\end{proof}

We denote by $F_\infty$ the homotopy colimit, in the model category $\MS(\Z)$, of the diagram
\[F_1\xrightarrow{\alpha}F_2\xrightarrow{\alpha}\ldots\]
where $\alpha:F_d\to F_{d+1}$ is the map constructed above. 

\begin{prop}\label{prop: stable homology}
The composite $C_d\to F_d\to F_{\infty}$ induces an isomorphism 
\[\H^i(F_\infty,\Z(q))\to\H^i(C_d,\Z(q))\]
for any $q$ and $i<l(d)$.
\end{prop}

\begin{proof}
By Proposition \ref{prop: stability for F} and an argument analogous to the previous corollary, we see that the map $F_d\to F_\infty$ induces an isomorphism
\[\H^i(F_\infty,\Z(q))\to\H^i(F_d,\Z(q))\]
for $i<d$ and any $q$. In the proof of Theorem \ref{theo: main bis}, we prove that the map
\[t_{<l(d)}\M(s_d):t_{<l(d)}\M(C_d)\to t_{<l(d)}\M(F_d)\]
is an isomorphism which again by the method of the previous corollary implies that the map
\[\H^i(F_d,\Z(q))\to\H^i(C_d,\Z(q))\]
is an isomorphism for $i<l(d)$. These two fact together imply the desired result.
\end{proof}

Note that the functor from smooth schemes over $\Z$ to topological spaces $X\mapsto X(\C)$ extends uniquely to a homotopy colimit preserving functor $\MS(\Z)\to \S$. We can thus make sense of the topological space $F_\infty(\C)$ and it is equivalent to the homotopy colimit of the diagram
\[F_1(\C)\xrightarrow{\alpha(\C)}F_2(\C)\xrightarrow{\alpha(\C)}
\ldots \xrightarrow{\alpha(\C)}F_n(\C)\xrightarrow{\alpha(\C)}\ldots\]

\begin{prop}\label{prop: Finfty=Omega0}
There is a weak equivalence $F_\infty(\C)\simeq \Omega_0^2S^2$
\end{prop}

\begin{proof}
The homotopy colimit diagram defining $F_\infty(\C)$ maps to a similar diagram
\[\Map_*(\P^1(\C),\P^1(\C))_1\to\Map_*(\P^1(\C),\P^1(\C))_2\to
\ldots \to\Map_*(\P^1(\C),\P^1(\C))_n\ldots\]
in which all the maps are weak equivalences. Moreover by \cite[Proposition 1.1.]{segaltopology}, the map $F_n(\C)\to \Map_*(\P^1(\C),\P^1(\C))_n$ induces an isomorphisms in a range of homotopy groups that grows to $\infty$ with $n$. It follows that the space $F_\infty(\C)$ is weakly equivalent to $\Map_*(\P^1(\C),\P^1(\C))_0$.
\end{proof}

\bibliographystyle{alpha}

\bibliography{biblio}

\end{document}